\documentclass[english,notitlepage]{amsart}
\usepackage[latin9]{inputenc}
\usepackage{geometry}
\usepackage{esint}
\usepackage{hyperref}
\usepackage{dsfont}

\usepackage{amssymb,amsmath,amsthm,amscd}
\usepackage{mathrsfs}
\usepackage{tikz}
\usetikzlibrary{positioning}
\usepackage{graphics}
\usetikzlibrary{matrix,arrows}
\usepackage{amssymb} 
\usepackage{babel}
\usepackage{amstext}
\usepackage{amscd}   
\usepackage{epsfig}  
\usepackage{rotating}
\usepackage{psfrag}  
\usepackage{multicol}
\usepackage{multirow}
\usepackage{array}
\usepackage{verbatim}

\makeatletter
\numberwithin{equation}{section}
\numberwithin{figure}{section}

\theoremstyle{plain}
\newtheorem{thm}{Theorem}
  \theoremstyle{plain}
  \numberwithin{thm}{section}
  \newtheorem{cor}[thm]{Corollary}
  \theoremstyle{plain}
  
  \theoremstyle{remark}
  \newtheorem{rem}[thm]{Remark}
    \theoremstyle{remark}

   \theoremstyle{plain}
  
  \def\Ddots{\mathinner{\mkern1mu\raise\p@
\vbox{\kern7\p@\hbox{.}}\mkern2mu
\raise4\p@\hbox{.}\mkern2mu\raise7\p@\hbox{.}\mkern1mu}}
\makeatother

\usepackage{babel}

\newtheorem{definition}{Definition}
\numberwithin{definition}{section}

\newcommand{\norm}[1]{\left\| #1 \right\|}
\newcommand{\mklm}[1]{\left\{ #1 \right\}}
\newcommand{\eklm}[1]{\left\langle #1 \right\rangle}

\renewcommand{\d}{\,d}
\newcommand{\N}{{\mathbb N}}
\newcommand{\Z}{{\mathbb Z}}
\newcommand{\C}{{\mathbb C}}
\newcommand{\Ccal}{{\mathcal C}}
\newcommand{\R}{{\mathbb R}}

\newcommand{\E}{{\mathcal E}}
\newcommand{\W}{{\mathcal W}}

\newcommand{\Jbb}{{\mathbb J }}
\newcommand{\M}{{\mathcal M}}

\newcommand{\Xbf}{{\mathbf X}}

\newcommand{\1}{{\bf 1}}
\renewcommand{\epsilon}{\varepsilon}

\renewcommand{\rho}{\varrho}

\newcommand{\bdm}{\begin{displaymath}}
\newcommand{\edm}{\end{displaymath}}
\newcommand{\bq}{\begin{equation}}
\newcommand{\eq}{\end{equation}}
\newcommand{\bqn}{\begin{equation*}}
\newcommand{\eqn}{\end{equation*}}

\newcommand{\Cinft}{{\rm C^{\infty}}}
\newcommand{\CT}{{\rm C^{\infty}_c}}

\renewcommand{\L}{{\rm L}}
\newcommand{\Lcal}{{\mathcal L}}
\newcommand{\Ncal}{{\mathcal N}}

\renewcommand{\S}{{\mathcal S}}

\newcommand{\g}{{\bf \mathfrak g}}

\renewcommand{\t}{{\bf \mathfrak t}}

\newcommand{\Ad}{\mathrm{Ad}\,}

\renewcommand{\det}{\mathrm{det}\,}

\newcommand{\vol}{\text{vol}\,}

\newcommand{\Crit}{\mathrm{Crit}}

\DeclareMathOperator{\supp}{supp\,}

\DeclareMathOperator{\mult}{mult}
\DeclareMathOperator{\tr}{tr}
\DeclareMathOperator{\gd}{\partial}

\begin{document}

\author{Pablo Ramacher}
\title[Addendum to ''Singular equivariant asymptotics and Weyl's law'']{Addendum to \\ ''Singular equivariant asymptotics and Weyl's law''} 
\address{Pablo Ramacher, Philipps-Universit\"at Marburg, Fachbereich Mathematik und Informatik, Hans-Meer\-wein-Str., 35032 Marburg, Germany}
\email{ramacher@mathematik.uni-marburg.de}
\date{\today}

\begin{abstract}
Let $M$ be  a closed Riemannian manifold  carrying an effective and isometric action of a compact connected Lie group $G$. We derive a refined remainder estimate in the stationary phase approximation of certain oscillatory integrals on $T^\ast M \times G$ with singular critical sets that were examined  in \cite{ramacher10} in order  to determine the asymptotic distribution of  eigenvalues  of an invariant elliptic   operator on $M$. As an immediate consequence, we deduce from this   an asymptotic multiplicity formula for families of irreducible representations in $\L^2(M)$. In forthcoming papers, the improved remainder will be used to prove an equivariant semiclassical Weyl law  \cite{kuester-ramacher15a} and a corresponding equivariant quantum ergodicity theorem \cite{kuester-ramacher15b}.  

\end{abstract}

\maketitle

\setcounter{tocdepth}{1}

\tableofcontents

\section{Introduction}

Let $M$ be a compact $n$-dimensional  Riemannian manifold without boundary,  carrying an isometric and  effective action of a connected compact Lie group $G$ with Lie algebra $\g$. In the  study of  the spectral geometry of $M$ one is led to an examination of the asymptotic behavior of oscillatory integrals of the form 
\begin{equation}
I(\mu):=\intop_{T^{*}U}\intop_{G}e^{i\mu\Phi(x,\xi,g)}a_\mu(x,\xi,g)\, dg\, d\left(T^{*}U\right)(x,\xi), \qquad \mu \to +\infty, \label{eq:integral}
\end{equation}
where $(\gamma,U)$ denotes a chart on $M$, $a_\mu\in \CT( T^{*}U\times G)$  an amplitude that might depend on the parameter  $\mu>0$, 
and the phase function is given by 
\bqn
\Phi(x,\xi,g):=\left\langle \gamma(x)-\gamma(g\cdot x ),\xi\right\rangle, \qquad (x,\xi) \in T_x^\ast U,\, g \in G,
\eqn
see   \cite{ramacher10, paniagua-ramacher12}. Here  $dg$ stands for the normalized Haar measure on $G$, and $d(T^\ast U)$  for the canonical symplectic volume form of the  co-tangent bundle of $U$, which coincides with the Riemannian volume form given by the Sasaki metric on $T^\ast U$. It is assumed that $(x,\xi,g) \in \supp a_\mu$ implies $g \cdot x \in U$, where we  wrote  $(x,\xi)$ for an element in $T^*U \simeq U \times \R^n$ with respect to the canonical trivialization of the co-tangent  bundle over the chart domain. The phase function $\Phi$ represents a global analogue of the momentum map  $\Jbb:T^\ast M \to \g^\ast$  of the Hamiltonian action of $G$ on $T^\ast M$, and oscillatory integrals with phase function given by the latter have been examined in \cite{ramacher13} in the context of equivariant cohomology.  The critical set of $\Phi$ is given by % \cite[Equation following (3.3)]{ramacher10}
 \begin{align*}
\Crit(\Phi)&=\mklm{(x,\xi,g) \in T^\ast U \times G : (\Phi_{\ast})_{(x,\xi,g)}=0}=\mathcal{C}\cap (T^\ast U\times G),
\end{align*}
where 
\bqn
\mathcal{C}:=\{(x,\xi,g)\in\Omega\times G:\; g\cdot(x,\xi)=(x,\xi)\}, \qquad \Omega:=\Jbb^{-1}(0).
\eqn
 Now, unless the $G$-action on $T^\ast M$ is free, the momentum map $\Jbb$ is not a submersion, so that the zero set $\Omega$ of the momentum map and the critical set of  $\Phi$ are not smooth manifolds. The stationary phase theorem can therefore not immediately be applied to the integrals  $I(\mu)$. Nevertheless, it was shown in \cite{ramacher10}   that by constructing a strong resolution of the set 
$
\Ncal:=\mklm{(x,g) \in M \times G : g\cdot x = x}
$
a partial desingularization $\mathcal{Z}: \widetilde {\bf X} \rightarrow {\bf X}:= T^\ast M \times G $ of the critical  set $\mathcal{C}$ can be achieved, and applying the stationary phase theorem in the resolution space $\widetilde {\bf X}$ an asymptotic description of $I(\mu)$ can be obtained. Indeed, the map $\mathcal{Z}$ yields 
a partial monomialization of  the local ideal $I_\Phi=(\Phi)$ generated by the phase function $\Phi$ according to  
\bqn 
\mathcal{Z}^\ast (I_\Phi) \cdot \E_{\tilde x, \widetilde {\bf X}} = \prod_j\sigma_j^{l_j}   \cdot\mathcal{Z}^{-1}_\ast(I_\Phi) \cdot \E_{\tilde x, \widetilde {\bf X}},
\eqn
 where $\E_{\widetilde{\bf X}}$ denotes the structure sheaf of rings of $\widetilde{\bf X}$,  $\mathcal{Z}^\ast (I_\Phi)$ the total transform, and $\mathcal{Z}_\ast^{-1} (I_\Phi)$ the weak transform of $I_\Phi$, while the  $\sigma_j$ are local coordinate functions near each $\tilde x \in \widetilde {\bf X}$ and the $l_j$ natural numbers. As a consequence, the phase function factorizes locally according to $\Phi \circ \mathcal{Z} \equiv \prod \sigma_j^{l_j} \cdot  \tilde \Phi^ {wk}$,
and one shows  that  the weak transforms $ \tilde \Phi^ {wk}$ have clean critical sets. 
Asymptotics for the integrals $I(\mu)$ are  then obtained by pulling  them back to the resolution space $\widetilde {\bf X}$, and applying  the stationary phase theorem to the $\tilde \Phi^{wk}$  with the variables  $\sigma_j$ as parameters. 
More precisely, let   $\text{Reg }\mathcal{C}\subset T^*M\times G$ denote the regular part of $\Ccal $, and regard it as a Riemannian submanifold with Riemannian metric  induced by the product metric of the Sasaki metric on $T^*M$ and some left-invariant Riemannian metric on $G$. The corresponding induced Riemannian volume form will be denoted by $ d(\textrm{Reg}\,\mathcal{\mathcal{C}})$. It was then shown in \cite[Theorem 9.1]{ramacher10} that for a $\mu$-independent amplitude $a\in \CT( T^{*}U\times G)$
one has the asymptotic formula
\bqn
I(\mu)=\left(\frac{2\pi}{\mu}\right)^{\kappa}\intop_{\textrm{Reg} \, {\mathcal{C}}}\frac{a(x,\xi,g)}{\left|\det\Phi''(x,\xi,g)_{|N_{(x,\xi,g)}\mathrm{Reg}\, {\mathcal{C}}}\right|^{1/2}}\, d(\textrm{Reg}\,\mathcal{\mathcal{C}})(x,\xi,g)+ \mathfrak{R}(\mu), 
\eqn
where $\kappa$ stands for the dimension of a $G$-orbit of principal type in $M$, $\Phi''(x,\xi,g)_{|N_{(x,\xi,g)}\mathrm{Reg}\,\mathcal{\mathcal{C}}}$
denotes the restriction of the Hessian of $\Phi$ to the normal space of $\mathrm{Reg}\,\mathcal{\mathcal{C}}$ inside $T^{*}U\times G$ at the point $(x,\xi,g)$, and the remainder satisfies the estimate
\bq
\label{eq:11.07.2015}
\mathfrak{R}(\mu)=O( \mu^{-\kappa-1}(\log\mu)^{\Lambda-1}),
\eq
 $\Lambda$ being the maximal number of elements of a totally ordered subset of the set of isotropy types of the $G$-action on $M$. \\

The goal of this note is to extend the asymptotic formula above  to amplitudes $a_\mu \in \CT(T^{*}U\times G) $ that depend on $\mu$, and  derive a refined remainder estimate of the form 
\begin{align}
\label{eq:12.07.2015}
 |\mathfrak{R}(\mu)| \leq %& \, C \norm{ |\det \psi''_{|N\mathrm{Reg}\, \Ccal}|^{-1/2} \cdot \big \|(\psi''_{|N\mathrm{Reg}\, \Ccal})^{-1}\big \|}_{\L^1(\mathrm{Reg} \, \Ccal\cap \supp a_\mu)} \\ & \cdot
 C\,    \sup_{l \leq 2\kappa+3} \norm{D^l a_\mu}_{\infty}  \,  \mu^{-\kappa-1}(\log\mu)^{\Lambda-1}
\end{align}
with a constant $C>0$ and differential operators $D^l$ of order $l$ independent of $\mu$ and $a_\mu$, where $\norm{ \cdot}_\infty$ denotes the supremum norm. This is accomplished in Theorem \ref{thm:main}. 

Integrals of the form $I(\mu)$ with $\mu$-dependent amplitudes occur in several contexts, and the bound \eqref{eq:12.07.2015} allows a precise control of the contributions of such amplitudes to the remainder $\mathfrak{R}(\mu)$.  As an immediate consequence of Theorem \ref{thm:main} we are able to prove an asymptotic multiplicity formula for families of unitary irreducible $G$-representations in the Hilbert space $\L^2(M)$ of square integrable functions on $M$, see Theorem \ref{thm:weylfam}, generalizing the  Weyl law for the reduced spectral counting function of an invariant elliptic operator proven in \cite{ramacher10} to families of representations whose cardinality is allowed to grow in a moderate way as the energy goes to infinity. Further applications will be given in \cite{kuester-ramacher15a} and \cite{kuester-ramacher15b}, where the refinement \eqref{eq:12.07.2015} will be  crucial  to prove a sharpened equivariant semiclassical Weyl law and a corresponding equivariant quantum ergodicity theorem, which otherwise could not be obtained. \\

\section{Refined singular equivariant asymptotics}

Let the notation and assumptions be as in the introduction. The main purpose of this note is  to prove the following
\begin{thm}
\label{thm:main}
Let $\mathcal{K}\subset T^\ast U \times G$ be a compactum and $a_\mu \in \CT(T^\ast U \times G)$  a family of amplitudes such that for each $\mu>0$ the support  of $a_\mu$  is contained in  $\mathcal{K}$.  Then, as $\mu\to+\infty$ one has the asymptotic formula  
\begin{align}
\begin{split}
\Big | I(\mu)&-\left(\frac{2\pi}{\mu}\right)^{\kappa}\intop_{\textrm{Reg} \, {\mathcal{C}}}\frac{a_\mu(x,\xi,g)}{\left|\det\Phi''(x,\xi,g)_{|N_{(x,\xi,g)}Reg\, {\mathcal{C}}}\right|^{1/2}}\, d(\textrm{Reg}\,\mathcal{\mathcal{C}})(x,\xi,g)\Big | \\
&\leq C  \sup_{l \leq 2\kappa+3} \norm{D^l a_\mu}_{\infty}  \mu^{-\kappa-1}(\log\mu)^{\Lambda-1},\label{eq:I(mu)}
\end{split}
\end{align}
where  the differential operators $D^l$  and the constant $C>0$ do not  depend on $\mu$ nor on $a_\mu$.  
\end{thm}

\begin{rem}
Note that both the derivatives and the supports of the amplitudes $a_\mu$  have to satisfy suitable assumptions about the way they depend on $\mu$ in order to obtain meaningful asymptotics.
\end{rem}

\begin{proof}
For $\mu$-independent amplitudes the statement of the theorem is  essentially the content of \cite[Theorem 9.1]{ramacher10}. In case of a $\mu$-dependent amplitude $a_\mu$, the precise dependence of the remainder estimate on the amplitudes $a_\mu$ has to be taken into account.  For this, let us  recall the main arguments and results from \cite{ramacher10}, and consider the decomposition of $M$ into orbit types
 \bq
 \label{eq:2.19}
M=M(H_1) \, \dot \cup \, \cdots \, \dot \cup \, M(H_L),
\eq
where we suppose that  the isotropy types $(H_1), \dots, (H_L)$ are numbered in such a way that $(H_i) \geq (H_j)$ implies $i \leq j$, compare Figure \ref{fig:tree}. One then constructs a partial desingularization 
\bq
\label{eq:220215}
\mathcal{Z}: \widetilde {\bf X} \rightarrow {\bf X}:= T^\ast M \times G 
\eq
of the critical set $\Ccal$ as follows. For each  $1 \leq N \leq \Lambda-1$ and each maximal, totally ordered subset $\mklm{(H_{i_1}),\dots, (H_{i_N})}$ of non-principal isoptropy types one constructs sequences of consecutive local blow-ups $\mathcal{Z}_{i_1\dots i_N}^{\rho_{i_1}\dots \rho_{i_N}}$ whose  respective centers are given by disjoint unions over maximal singular isotropy bundles 
\begin{figure}[h!]
\begin{tikzpicture}[node distance=1.4cm, auto]

\node (A00) {$H_{L}$}; 

\node (B0) [above left of=A00] {$H_{L-4}$}; 
\node (C0) [right of=B0] {$H_{L-3}$}; 
\node (D0) [right of=C0] {$H_{L-2}$}; 
\node (F0) [right of=D0] {$H_{L-1}$}; 

\node (D) [above left of=B0] {$H_{m-1}$}; 
\node (E) [right of=D] {$H_m$}; 
\node (F) [right of=E] {$H_{m+1}$}; 

\node (E1) [above left of =D] {$H_{i+2}$}; 
\node (F1) [right of=E1] {$H_{i+3}$}; 
\node (G1) [right of=F1] {$\cdots$};
\node (H1) [right of=G1] {$H_l$};
\node (I1) [right  of=H1] {$H_{l+1}$};

\node (D2) [above left of=E1] {$H_1$}; 
\node (E2) [right of=D2] {$H_2$}; 
\node (F2) [right of=E2] {$H_3$}; 
\node (G2) [right of=F2] {$\cdots$};
\node (H2) [right of=G2] {$H_{i-1}$};
\node (I2) [right of=H2] {$H_{i}$};
\node (J2) [right of=I2] {$H_{i+1}$};

\draw[-] (A00) to node {} (B0);
\draw[-] (A00) to node {} (C0);
\draw[-] (A00) to node {} (D0);
\draw[-] (A00) to node {} (F0);

\draw[-] (E1) to node {} (D2);
\draw[-] (E2) to node {} (E1);
\draw[-] (E2) to node {} (F1);

\draw[-, dashed] (D) to node {} (B0);
\draw[-, dashed] (E) to node {} (C0);
\draw[-, dashed] (F) to node {} (C0);

\draw[-] (F1) to node {} (F2);
\draw[-] (E2) to node {} (E1);

\draw[-] (H1) to node {} (H2);
\draw[-] (I2) to node {} (I1);
\draw[-] (J2) to node {} (I1);

\draw[-, dashed] (F) to node {} (H1);
\draw[-, dashed] (F) to node {} (I1);

\draw[-, dashed] (D) to node {} (E1);
\draw[-, dashed] (D) to node {} (F1);

\draw[-, dashed] (I1) to node {} (F0);

\end{tikzpicture}
\caption{An isotropy tree corresponding to the decomposition \eqref{eq:2.19}.  A line between two subgroups indicates partial ordering.}\label{fig:tree}
\end{figure}
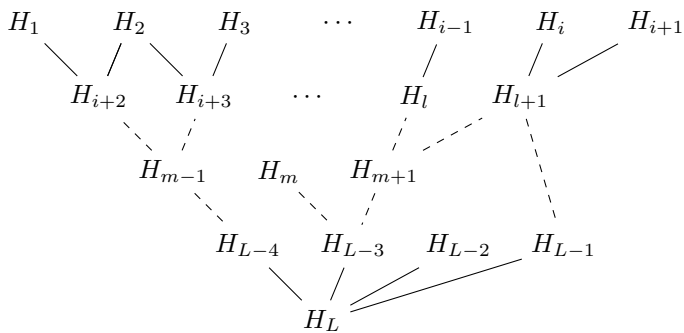
labeled by the types $\mklm{(H_{i_1}),\dots, (H_{i_N})}$, and are realized in  a set of local charts labeled by the indices $\rho_{i_1},\dots, \rho_{i_N}$, see \cite[Eq. (6.7)]{ramacher10}. The global morphism induced by the local transformations $\mathcal{Z}_{i_1\dots i_N}^{\rho_{i_1}\dots \rho_{i_N}}$  is then denoted by \eqref{eq:220215}. For the precise construction,  the reader is referred to   \cite[Beginning of Section 9 and  Section 5]{ramacher10}. Now, if we pull  the phase function $\Phi$ back along   the maps  $\mathcal{Z}^{\rho_{i_1}\dots \rho_{i_N}}_{i_1\dots i_N}$, it factorizes locally according to 
\bqn
\Phi \circ \mathcal{Z}^{\rho_{i_1}\dots \rho_{i_N}}_{i_1\dots i_N} =\, ^{(i_1\dots i_N)} \widetilde \Phi^{tot}=\tau_{i_1}(\sigma) \dots \tau_{i_N}(\sigma) \, ^{(i_1\dots i_N)}\widetilde \Phi^ {wk}, 
\eqn
 where the  $\tau_{i_j}$ are monomials in the exceptional parameters $\sigma_{i_1},\dots, \sigma_{i_N}$ given by the sequence of local quadratic transformations
 \begin{align*}
\delta_{i_1\dots i_N}: (\sigma_{i_1}, \dots \sigma_{i_N}) &\mapsto \sigma_{i_1}( 1, \sigma_{i_2}, \dots, \sigma_{i_N})= (\sigma_{i_1}', \dots ,\sigma_{i_N}')\mapsto \sigma_{i_2}'(\sigma_{i_1}',1,\dots, \sigma_{i_N}')= (\sigma_{i_1}'', \dots, \sigma_{i_N}'')\\
 &\mapsto \sigma_{i_3}''(\sigma_{i_1}'',\sigma_{i_2}'', 1,\dots, \sigma_{i_N}'')= \cdots \mapsto \dots = (\tau_{i_1}, \dots ,\tau_{i_N}).
\end{align*}
On the other hand, if we  transform the oscillatory integral $I(\mu)$ under the global morphism  $\mathcal{Z}$ using suitable partitions of unity we obtain the decomposition
\begin{align}
\label{eq:65}
I(\mu)&=\sum _{N=1}^{\Lambda-1}\, \sum_{\stackrel{i_1<\dots< i_N}{ \rho_{i_1}, \dots ,\rho_{i_N}}} I_{i_1\dots i_N}^{\rho_{i_1} \dots \rho_{i_{N}}}( \mu)+\sum_{N=1}^{\Lambda-1} \, \sum_{\stackrel{j_1<\dots< j_{N-1}<L}{ \rho_{j_1}, \dots ,\rho_{j_{N-1}}}} I_{j_1\dots j_{N-1} L}^{\rho_{j_1} \dots \rho_{j_{N-1}}}( \mu) + R(\mu),
\end{align}
where the first term is a sum over  maximal, totally ordered subsets of non-principal isotropy types, while the second term is given by a sum over arbitrary,  totally ordered subsets of non-principal isotropy types, and $R(\mu)$ denotes certain non-stationary contributions, see \cite[Eq. (9.1)]{ramacher10}. Here  
\bqn
\label{eq:fub}
I_{i_1\dots i_N}^{\rho_{i_1}\dots \rho_{i_N}}(\mu):= \int_{ (-1,1)^N}    J_{i_1\dots i_N}^{\rho_{i_1}\dots \rho_{i_N}}\Big (  \mu \cdot {\tau_{i_1}(\sigma) \cdots \tau_{i_N}(\sigma) } \Big ) \prod_{j=1}^N |\tau_{i_j}(\sigma)|^{e^{(i_j)}-1}  |\det D\delta_{i_1\dots i_N}(\sigma)| \d \sigma,
\eqn
and 
\begin{align*}
\label{eq:62}
\begin{split}
  J_{i_1\dots i_N}^{\rho_{i_1}\dots \rho_{i_N}}(\nu):= \int_{\widetilde \Xbf_{i_1\dots i_N}^{\rho_{i_1}\dots \rho_{i_N}}}  e^{i\nu \, ^{(i_1\dots i_N)} \widetilde \Phi_\sigma ^{wk}}  \, a_{i_1\dots i_N}^{\rho_{i_1}\dots \rho_{i_N}}\,  {\mathcal{J}}_{i_1\dots i_N}^{\rho_{i_1}\dots \rho_{i_N}},
    \end{split}
\end{align*}
where the $\widetilde \Xbf_{i_1\dots i_N}^{\rho_{i_1}\dots \rho_{i_N}}$ are suitable submanifolds in the resolution space $\widetilde \Xbf$ of codimension $N$, the amplitudes $a_{i_1\dots i_N}^{\rho_{i_1}\dots \rho_{i_N}}$ are compactly supported and given by pullbacks of  $a_\mu$ multiplied by  elements of partitions of unity, the  ${\mathcal{J}}_{i_1\dots i_N}^{\rho_{i_1}\dots \rho_{i_N}}$ are  Jacobians independent of $\mu$, and the $e^{(i_j)}$ are natural numbers satisfying $e^{(i_j)}-1  \geq \kappa$, see \cite[Section 8, in particular Lemma 8.1]{ramacher10}.   Besides,  we introduced the new parameter
\bqn
\nu := \mu\cdot  {\tau_{i_1}(\sigma) \cdots \tau_{i_N}}(\sigma), 
\eqn
and wrote  $ ^{(i_1\dots i_N)} \widetilde \Phi ^{wk}_\sigma $ for  the weak transform $ ^{(i_1\dots i_N)} \widetilde \Phi ^{wk}$ regarded as a function on $\widetilde \Xbf_{i_1\dots i_N}^{\rho_{i_1}\dots \rho_{i_N}}$, while the variables $\sigma=(\sigma_{i_1},\dots,  \sigma_{i_N})$ are regarded as parameters. It can then be shown that  the weak transforms $\, ^{(i_1\dots i_{N})}\widetilde \Phi^ {wk}$ have clean critical sets on $(-1,1)^{N} \times \widetilde \Xbf_{i_1\dots i_{N}}^{\rho_{i_1}\dots \rho_{i_{N}}}$ of codimension $2\kappa$ \cite[Theorems 6.1 and 7.1]{ramacher10}. By transversality, this implies that the phase functions $ ^{(i_1\dots i_N)} \widetilde \Phi ^{wk}_\sigma $ have clean critical sets, too.  Similarly,
\bqn 
I_{j_1\dots j_{N-1}j_N}^{\rho_{j_1}\dots \rho_{j_{N-1}}}(\mu):= \int_{ (-1,1)^{N-1}}    J_{j_1\dots j_{N-1}j_N}^{\rho_{j_1}\dots \rho_{j_{N-1}}}\Big (  \mu \cdot {\tau_{j_1}(\sigma) \cdots \tau_{j_{N-1}}(\sigma) } \Big ) \prod_{k=1}^{N-1} |\tau_{j_k}(\sigma)|^{e^{(j_k)}-1}  |\det D\delta_{j_1\dots j_{N-1}}(\sigma)| \d \sigma,
\eqn
where now
\begin{align*}
\begin{split}
  J_{j_1\dots j_{N-1}j_N}^{\rho_{j_1}\dots \rho_{j_{N-1}}}(\nu):= \int_{\widetilde \Xbf_{j_1\dots j_{N-1}}^{\rho_{j_1}\dots \rho_{j_{N-1}}}(H_{j_N})}  e^{i\nu \, ^{(j_1\dots j_{N-1})} \widetilde \Phi_\sigma ^{wk}}  \, a_{j_1\dots j_{N-1}}^{\rho_{j_1}\dots \rho_{j_{N-1}}}\,  {\mathcal{J}}_{j_1\dots j_{N-1}}^{\rho_{j_1}\dots \rho_{j_{N-1}}} \, \chi_{j_1\dots j_{N-1}j_N},
    \end{split}
\end{align*}
 ${\chi_{j_1\dots j_{N-1}j_N}}$ being a cut-off-function with compact support in  $\widetilde \Xbf_{j_1\dots j_{N-1}}^{\rho_{j_1}\dots \rho_{j_{N-1}}}(H_{j_N})$. In case that $j_N=L$,  $\, ^{(j_1\dots j_{N-1})}\widetilde \Phi^ {wk}$  has a clean critical set on $(-1,1)^{N-1} \times \widetilde \Xbf_{j_1\dots j_{N-1}}^{\rho_{j_1}\dots \rho_{j_{N-1}}}(H_L)$ of codimension $2\kappa$, compare \cite[Lemma 7.3]{ramacher10}. 

Now, let us proceed to the proof of   \eqref{eq:I(mu)}, for which we intend  to apply the stationary phase principle to the partially factorized phase functions. In order to deal with the competing asymptotics $\mu \to \infty$ and $\tau_{i_j} \to 0$, we define for a sufficiently small  $\epsilon>0 $ to be chosen later the integrals
\begin{align*}
\, ^1I_{i_1\dots i_N}^{\rho_{i_1}\dots \rho_{i_N}}(\mu)&:= \int_{\epsilon<|\tau_{i_j}(\sigma)|<1}      J_{i_1\dots i_N}^{\rho_{i_1}\dots \rho_{i_N}}\Big (  \mu \cdot {\tau_{i_1}(\sigma) \cdots \tau_{i_N}(\sigma) } \Big ) \prod_{j=1}^N |\tau_{i_j}(\sigma)|^{e^{(i_j)}-1}  |\det D\delta_{i_1\dots i_N}(\sigma)| \d \sigma, \\
\, ^2 I_{i_1\dots i_N}^{\rho_{i_1}\dots \rho_{i_N}}(\mu)&:= \int_{0<|\tau_{i_j}(\sigma)|<\epsilon}   J_{i_1\dots i_N}^{\rho_{i_1}\dots \rho_{i_N}}\Big (  \mu \cdot {\tau_{i_1}(\sigma) \cdots \tau_{i_N}(\sigma) } \Big ) \prod_{j=1}^N |\tau_{i_j}(\sigma)|^{e^{(i_j)}-1}  |\det D\delta_{i_1\dots i_N}(\sigma)| \d \sigma. 
\end{align*}
Now,  since the desingularization \eqref{eq:220215} of the critical set $\Ccal$ is based on a strong resolution of the set $\mathcal{N}=\mklm{(x,g) \in M \times G: \, g\cdot x=x}$,  the variable $\xi$ is not affected by the resolution process, so that  $\widetilde \Xbf_{i_1\dots i_N}^{\rho_{i_1}\dots \rho_{i_N}}=\widetilde {\bf Y}_{i_1\dots i_N}^{\rho_{i_1}\dots \rho_{i_N}}\times \R^n_\xi$, where the $\widetilde {\bf Y}_{i_1\dots i_N}^{\rho_{i_1}\dots \rho_{i_N}}$ are sub-manifolds with compact closure, and consequently of finite volume, compare \cite[Eq. (8.1)]{ramacher10}.   It is then immediate that
\begin{align}
\label{eq:I2}
\begin{split}
\big | \, ^2 I_{i_1\dots i_N}^{\rho_{i_1}\dots \rho_{i_N}}(\mu)\big | &\leq  \int_{(-\epsilon,\epsilon)^N}  \norm{a_{i_1\dots i_N}^{\rho_{i_1}\dots \rho_{i_N}}\,  {\mathcal{J}}_{i_1\dots i_N}^{\rho_{i_1}\dots \rho_{i_N}}}_{\infty}  \vol_{\widetilde \Xbf_{i_1\dots i_N}^{\rho_{i_1}\dots \rho_{i_N}}} (\supp a_{i_1\dots i_N}^{\rho_{i_1}\dots \rho_{i_N}} ) \prod_{j=1}^N |\tau_{i_j}|^{e^{(i_j)} -1} \d \tau_{i_N} \dots \d \tau_{i_1} \\
& \leq C \,   \norm{a_\mu}_\infty  \int_{(-\epsilon,\epsilon)^N}  \prod_{j=1}^N |\tau_{i_j}|^{\kappa} \d \tau_{i_N} \dots \d \tau_{i_1} =\frac{2C}{\kappa+1}  \norm{a_\mu}_\infty  \epsilon^{N (\kappa+1)}
\end{split}
\end{align}
for some constant $C>0$ independent of $\mu$ and $a_\mu$, since we assumed that  $\supp a_\mu \subset \mathcal{K}$ for all $\mu>0$. Let us now turn to the integrals $\, ^1I_{i_1\dots i_N}^{\rho_{i_1}\dots \rho_{i_N}}(\mu)$.  As in \cite[Theorem 8.2]{ramacher10}, the generalized stationary phase theorem \cite[Theorem 4.1]{ramacher10} yields
for fixed $\sigma=(\sigma_{i_1},\dots, \sigma_{i_N})$  and arbitrary $\widetilde N\in \N$ the asymptotic expansion
\bq
\label{eq:2.21}
 J_{i_1\dots i_N}^{\rho_{i_1}\dots \rho_{i_N}}(\nu)=(2\pi |\nu|^{-1})^{\kappa} \sum_{j=0}^{\widetilde N-1} Q_{j,\sigma} |\nu|^{-j}+ R_{\widetilde N,\sigma}(\nu),
\eq
together with the explicit estimates
\begin{align*}
\begin{split}
|Q_{j,\sigma}| &\leq \widetilde C_{j,  ^{(i_1\dots i_N)} \widetilde \Phi ^{wk}_\sigma} \vol_{\widetilde \Ccal_\sigma} (\supp a_{i_1\dots i_N}^{\rho_{i_1}\dots \rho_{i_N}} \cap \widetilde \Ccal_\sigma ) \sup_{l \leq 2j} \norm{D^l (a_{i_1\dots i_N}^{\rho_{i_1}\dots \rho_{i_N}}\,  {\mathcal{J}}_{i_1\dots i_N}^{\rho_{i_1}\dots \rho_{i_N}})}_{\infty}, \\
|R_{\widetilde N,\sigma}(\nu)|&\leq C_{\widetilde N,  ^{(i_1\dots i_N)} \widetilde \Phi ^{wk}_\sigma} \vol_{\widetilde \Xbf_{i_1\dots i_N}^{\rho_{i_1}\dots \rho_{i_N}}} (\supp a_{i_1\dots i_N}^{\rho_{i_1}\dots \rho_{i_N}} ) \sup_{l \leq 2\kappa +2\widetilde N+1 } \norm{D^l (a_{i_1\dots i_N}^{\rho_{i_1}\dots \rho_{i_N}}\,  {\mathcal{J}}_{i_1\dots i_N}^{\rho_{i_1}\dots \rho_{i_N}} )}_{\infty}\, \mu^{-\kappa-\widetilde N},
\end{split}
\end{align*}
where we wrote $\widetilde \Ccal_\sigma=  \mathrm{Crit} \, ^{(i_1\dots i_N)} \widetilde \Phi ^{wk}_\sigma$. 
 Moreover,   the constants  $ \widetilde C_{j,  ^{(i_1\dots i_N)} \widetilde \Phi ^{wk}_\sigma}$ and  $C_{\widetilde N,  ^{(i_1\dots i_N)} \widetilde \Phi ^{wk}_\sigma}$ are essentially bounded from above by  
 \bq
 \label{eq:06.07.2015}
 \sup_{\supp a_{i_1\dots i_N}^{\rho_{i_1}\dots \rho_{i_N}} \cap \, \widetilde \Ccal_\sigma } { |\det \mathrm{Hess} {\,  ^{(i_1\dots i_N)} \widetilde \Phi ^{wk}_\sigma}_{|N\widetilde  \Ccal}|^{-1/2} \cdot \big \|(\mathrm{Hess} {\,   ^{(i_1\dots i_N)} \widetilde \Phi ^{wk}_\sigma}_{|N\widetilde  \Ccal})^{-1}\big \|}^r
 \eq
 with $r=j,\widetilde N$, respectively, compare  \cite[Remark 4.2]{ramacher10}.  Now,  the main consequence to be drawn from  \cite[Theorems 6.1 and 7.1]{ramacher10} is that these bounds are uniform in $\sigma$, see \cite[Remark 8.3 and  proof of Theorem 8.2]{ramacher10}.  On the other hand, the amplitude $a_{i_1\dots i_N}^{\rho_{i_1}\dots \rho_{i_N}}$ is given in terms of $a_\mu$, so that the bounds  \eqref{eq:06.07.2015}  still depend on $\mu$. 
%But since we assumed that  $\supp a_\mu \subset \mathcal{K}$ for all $\mu>0$, \eqref{eq:06.07.2015} is uniformly bounded for all $\sigma $ \emph{and}  $\mu$.
%\footnote{Here, the assumption about the supports of the amplitudes $a_\mu$ is not needed, 
But  since the desingularization \eqref{eq:220215} does not affect the variable $\xi$, the weak transform  $\,  ^{(i_1\dots i_N)} \widetilde \Phi ^{wk}_\sigma$ depends linearly on $\xi$,  and its tranversal Hessian will depend  linearly on $\xi$ as well, see \cite[Page 54]{ramacher10}. From this one concludes that \eqref{eq:06.07.2015} is uniformly bounded for all $\sigma $ \emph{and}  $\mu$, since $\xi\in \R^n$  is the only variable in the resolution space  with a non-pre-compact domain of definition.
Taking into account that  $\supp a_\mu \subset \mathcal{K}$ for all $\mu>0$,  integration of \eqref{eq:2.21} now yields   for $\widetilde N=1$
\begin{align*}
\Big | \, ^1I_{i_1\dots i_N}^{\rho_{i_1}\dots \rho_{i_N}}(\mu)  &- (2\pi /\mu)^{\kappa}  \int _{\epsilon < |\tau_{i_j}(\sigma) |< 1}   Q_{0,\sigma} \prod_{j=1}^N |\tau_{i_j}(\sigma)|^{e^{(i_j)} -1-\kappa}  |\det D\delta_{i_1\dots i_N}(\sigma) | \d \sigma \Big| \\&
\leq c_1\,  \sup_{l \leq 2\kappa +3 } \norm{D^l a_\mu}_\infty  \mu^{-\kappa-1}   \int_{\epsilon < |\tau_{i_j}|< 1}   \prod_{j=1}^N |\tau_{i_j}|^{e^{(i_j)} -2-\kappa} \d \tau\\
& \leq c_2 \,  \sup_{l \leq 2\kappa +3 } \norm{D^l a_\mu}_\infty  \mu^{-\kappa-1} \, \prod _{j=1}^N  \max\{1,(-\log \epsilon)^{q_j}\}, 
\end{align*}
where the exponents $q_j$ can take the values $0$ or $1$, and the constants $c_i>0$ are independent of $\sigma$, $\mu$, and $a_\mu$.  Having in mind that we are interested in the case where $\mu \to +\infty$, we now set $
\epsilon=\mu^{-1/N}$. By taking into account \eqref{eq:I2} and the fact that 
\bqn 
(2\pi /\mu)^\kappa \int_{0 < |\tau_{i_j}(\sigma) |< \mu^{-1/N}}  Q_{0,\sigma}  \prod_{j=1}^N |\tau_{i_j}(\sigma)|^{e^{(i_j)} -1-\kappa}  |\det D\delta_{i_1\dots i_N}(\sigma) | \d \sigma \leq c_3 \,  \norm{a_\mu}_\infty \, \mu^{-\kappa-1}
\eqn
  we finally obtain for each of the integrals $I_{i_1\dots i_N}^{\rho_{i_1}\dots \rho_{i_N}}(\mu)$  the asymptotic expansion 
 \bqn 
I_{i_1\dots i_N}^{\rho_{i_1}\dots \rho_{i_N}}(\mu)=({2 \pi}/ \mu)^\kappa \Lcal_{i_1\dots i_N}^{\rho_{i_1}\dots \rho_{i_N}}+ c_4 \,  \sup_{l \leq 2\kappa +3 } \norm{D^l a_\mu}_\infty  \,\mu^{-\kappa-1}(\log \mu)^N,
\eqn
where the leading coefficient $\Lcal_{i_1\dots i_N}^{\rho_{i_1}\dots \rho_{i_N}}$ is given by 
\bqn
\Lcal_{i_1\dots i_N}^{\rho_{i_1}\dots \rho_{i_N}}:=\int_{\Crit( ^{(i_1\dots i_N)} \widetilde \Phi^{wk})} \frac { a_{i_1\dots i_N}^{\rho_{i_1}\dots \rho_{i_N}} {\mathcal{J}}_{i_1\dots i_N}^{\rho_{i_1}\dots \rho_{i_N}} \, d\Crit( ^{(i_1\dots i_N)} \widetilde \Phi^{wk})} {|\det \mathrm{Hess} ( ^{(i_1\dots i_N)} \widetilde \Phi^{wk})_{|N\Crit( ^{(i_1\dots i_N)} \widetilde \Phi^{wk})}|^{1/2}},
\eqn
and $ d\Crit( ^{(i_1\dots i_N)} \widetilde \Phi^{wk})$ denotes the induced Riemannian volume density. This is  \cite[Theorem 8.4]{ramacher10} with a refined  remainder estimate.
Since the integrals $I_{j_1\dots j_{N-1}L}^{\rho_{j_1}\dots \rho_{j_{N-1}}}(\mu)$ have analogous descriptions, we are left with the task of examining the non-stationary contributions $R(\mu)$ in \eqref{eq:65}. They are of two types: either they arise by localizing the integrals $ J_{j_1\dots j_{N-1}j_N}^{\rho_{j_1}\dots \rho_{j_{N-1}}}(\nu)$  to tubular neighborhoods of the relevant critical sets, or they correspond to integrals $J_{i_1\dots i_{N}}^{\rho_{i_1}\dots \rho_{i_{N}}}(\nu)$  over charts of the resolution spaces where the weak transforms of the phase functions do not have critical points. In both cases, the considered non-stationary domains do have $\mu$- and $a_\mu$-independent  distances strictly larger than zero to the relevant critical sets due to the particular resolution and cut off functions employed, see \cite[Pages 27, 34, and 38]{ramacher10}. Furthermore,  the $\xi$-gradients of the weak transforms are given in terms of those distances, compare \cite[Eq. (6.11)]{ramacher10}. 
Therefore, the  lengths of the gradients of the weak transforms of the non-stationary contributions are uniformly bounded from below on the pre-images of $\supp a_\mu$ in any of  the  relevant charts,  no matter how the support of $a_\mu$ varies as $\mu \to \infty$.  The non-stationary phase principle \cite[Page 19]{grigis-sjoestrand} then implies that they will contribute only terms of order
\bqn 
|R(\mu)| \leq c_5 \,  \sup_{l \leq \kappa +1 } \norm{D^l a_\mu}_\infty  \,\mu^{-\kappa-1},
\eqn
compare \cite[Page 62]{ramacher10}. 
By taking into account \eqref{eq:65} and the asymptotic descriptions of  $I_{i_1\dots i_{N}}^{\rho_{i_1}\dots \rho_{i_{N}}}(\mu)$ and
$I_{j_1\dots j_{N-1}L}^{\rho_{j_1}\dots \rho_{j_{N-1}}}(\mu)$, the assertion of the theorem follows, the computation of the leading term being already accomplished in \cite[Theorem 9.1]{ramacher10}. 
\end{proof}

\section{An asymptotic multiplicity formula in $\L^2(M)$}

\subsection{Asymptotic behavior of families of irreducible representations}\label{sec:famirreps}

Let the notation be as in the introduction and $\widehat G$  the set of all equivalence classes of unitary irreducible representations of $G$.  If $\chi \in \widehat G$ and $(\pi_\chi, H_\chi)\in \chi$,  $ H_\chi $ has finite dimension,  and the \emph{character of $\chi$} is  given by
\bqn 
\chi(g):=\tr \pi_\chi(g), \qquad g \in G.
\eqn
It is  denoted by the same letter. Let $d_\chi:=\chi(e)$ be the dimension of $\pi_\chi$.
Endow $M$ with  the Riemannian volume density $dM$, and consider the Peter-Weyl decomposition   
\bq
\label{eq:peter-weyl}
\L^2(M) = \bigoplus_{\chi \in \widehat G} \L^2_\chi(M)
\eq
 of the left-regular representation of $G$ in   $\L^2(M)$ into isotypic components. As an immediate consequence of Theorem \ref{thm:main} we shall generalize the  Weyl law for the reduced spectral counting function of an invariant elliptic operator on $M$ proven in \cite[Theorem 9.5]{ramacher10} to 
sums of isotypic components of the form $\bigoplus_{\chi \in \W_\lambda} \L^2_\chi(M)$, where  $\W_\lambda\subset \widehat G$ are  appropriate finite subsets  whose cardinality does not grow too fast  as $\lambda \to +\infty$. Thus, let 
  \bqn
 P_0: \Cinft(M) \longrightarrow \L^2(M)
 \eqn
be   an invariant elliptic classical pseudodifferential operator of order $m$ on $M$ with principal symbol $p(x,\xi)$, where $\Cinft(M)$ denotes the space of smooth functions on $M$. Assume that $P_0$ is positive and symmetric, and denote by $P$ its unique self-adjoint extension. Denote by ${\widehat G}'\subset \widehat G$ the subset of equivalence classes of representations occuring in \eqref{eq:peter-weyl}. Since $P$ commutes with the $G$-action, each of its eigenspaces constitutes an unitary $G$-module, and we write  $\mult_\chi(t)$ for the multiplicity of the unitary irreducible representation $(\pi_\chi, H_\chi)$ in the eigenspace $E_t$ of $P$ belonging to the eigenvalue $t$.  The following families of irreducible $G$-representations were first considered in \cite{kuester-ramacher15a} within a semiclassical context. 

\begin{definition}
\label{def:semiclfam} 
Let    $\{\W_\lambda\}_{\lambda\in (0,\infty)}$ be a family of finite sets $\W_{\lambda}\subset \widehat G'$ such that there is a $\vartheta\geq 0$  so that  for each $l\in \N$ and each differential operator $D^l$ on $G$ of order $l$ there is a constant $C>0$ independent of $\lambda$ with
\[
\max_{\chi\in \W_\lambda}\frac{\norm{D^l\, \overline{\chi}}_\infty}{\left[\pi_{\chi|_{H}}:\mathds{1}\right]}\leq C\,\lambda^{\vartheta l}\qquad \forall \;\lambda\in(0,\infty),
\]
where $ [{\pi_\chi}_{|H}:\1]$ denotes the multiplicity of the trivial representation in the restriction of $\pi_\chi$ to a principal isotropy group $H$. 
The smallest possible $\vartheta$ is called the \emph{growth rate} of the family $\W_\lambda$.
\end{definition}
\begin{rem}
Note that $[{\pi_\chi}_{|H}:\1]=\int_H \overline{\chi(h)} \d h$, see \cite[Eq. (3.33)]{cassanas}, and by the Frobenius reciprocity theorem one has $[{\pi_\chi}_{|H}:\1]=[\L^2(G/H):\pi_\chi]$.  Furthermore, the irreducible $G$-representations appearing in the Peter-Weyl decomposition of $\L^2(M)$ are  precisely those $G$-representations appearing in $\L^2(G/H)$, so that $[{\pi_\chi}_{|H}:\1]\geq 1$ for $\chi \in \widehat G'$. If the orbit space $M/G$ has dimension greater than one then each irreducible $G$-representation appears an infinite number of times, compare \cite[Section 2]{donnelly78}.
\end{rem}

The following result is a generalization of \cite[Theorem 9.5]{ramacher10} to growing families of isotypic components. 
\begin{thm}\label{thm:weylfam}
With the notation as above assume that $n-\kappa \geq 1$, set 
\bqn
\M_\chi(\lambda):=\sum _{t \leq \lambda} \mult_\chi(t),
\eqn 
and write $S^\ast M:= \mklm{(x,\xi) \in T^\ast M: p(x,\xi) = 1}$. If $\W_\lambda\subset \widehat G'$ is a family of growth rate $\vartheta\in \big [0, \frac 1{(2\kappa+3)m}\big )$, then as $\lambda \to +\infty$
\bqn 
\frac 1{|\W_\lambda|}\sum_{\chi \in \W_\lambda} \frac{\M_\chi(\lambda)}{[{\pi_\chi}_{|H}:\1]}= \frac{\mathrm{vol} \, [(\Omega \cap S^\ast M)/G]}{(n-\kappa)(2\pi)^{n-\kappa}}    
 \,  {\lambda} ^{\frac{n-\kappa}m }   + O\big (\lambda^{\frac{n-\kappa-1}m+ \vartheta(2\kappa+3)} (\log \lambda )^\Lambda \big ).
\eqn
%,  while $\Omega=\mathbb{J}^{-1}(0)$ denotes the zero level of the momentum map $\mathbb{J}:T^\ast M \rightarrow \g^\ast$ of the Hamiltonian $G$-action on $T^\ast M$.
 \end{thm}

\begin{proof} 
Let $Q=(P)^{1/m}$ be the $m$-th root of $P$ given by the spectral theorem. It is a classical pseudodifferential operator of order $1$ with principal symbol $q(x,\xi)=p(x,\xi)^{1/m}$,  and if $0<\lambda_1 \leq \lambda_2 \leq \dots$ are  the eigenvalues of $P$ repeated according to their multiplicity, the eigenvalues of $Q$ are $\mu_j :=(\lambda_j)^{1/m}$. The operator $Q$ is invariant,  and we write ${\mult^Q_\chi(\mu_j)}$ for  the multiplicity of the representation $\pi_\chi$ in the eigenspace $E^Q_{\mu_j}$  of $Q$ belonging to the eigenvalue $\mu_j$. Put $\M_\chi^Q(\mu):=\sum _{t \leq \mu} \mult^Q_\chi(t)$.
As explained in \cite[Section 2]{ramacher10}, an asymptotic description of the reduced spectral counting function $N_\chi(\lambda):=d_\chi \M_\chi(\lambda)=d_\chi \M_\chi^Q(\lambda^{1/m})$ can be obtained by studying the singularities of the tempered distribution
$$
\sum _{j=1}^\infty m^Q_\chi(\mu_j)\,  e^{-it\mu_{j}} \in \S'(\R), \qquad m^Q_\chi(\mu_j):=d_\chi {\mult^Q_\chi(\mu_j)}/{\dim E^Q_{\mu_j}}.
$$
It corresponds to the distribution trace of $P_\chi \circ U(t)$,  where $U(t):=e^{-itQ}$ denotes the Fourier transform of the spectral measure of $Q$  and $P_\chi$  the projector onto the isotypic component $\L_\chi^2(M)$. This distribution trace is the Fourier transform of the distribution
\bqn 
\sigma_\chi(\mu):=\sum _{j=1}^\infty m^Q_\chi(\mu_j)\, \delta(\mu- \mu_{j}),
\eqn
and the main singularity of $\hat \sigma_\chi$ at $t=0$ can be described by approximating $U(t)$ via Fourier integral operators, yielding for sufficiently small $\delta >0$ and $\mu \geq 1$ a description  of
\bqn 
{\hat \sigma_\chi (\check \rho e^{i(\cdot) \mu})} =\sum_{j=1}^\infty m^Q_\chi(\mu_j) \hat \rho (\mu-\mu_j), \qquad \rho \in \CT(-\delta,\delta),
\eqn 
by  oscillatory integrals of the form \eqref{eq:integral}. Indeed, consider an atlas $\mklm{(\kappa_\gamma, U_\gamma)}$ for $M$ with subordinated partition of unity $\mklm{f_\gamma}$, and additional test functions $\bar f_\gamma \in \CT(U_\gamma)$ satisfying $\bar f_\gamma \equiv 1$ on $\supp f_\gamma$. Let further $\rho \in \CT(-\delta,\delta)$ be such that $1=2\pi \rho(0)$, and $0 \leq \alpha  \in \CT(1/2,3/2)$  satisfy $\alpha \equiv 1$ in a neighborhood of $1$, while $\Delta_{\epsilon, r} \in \CT(\R)$ denotes an approximation of the $\delta$-distribution at $r \in \R$ as $\epsilon \to 0$. Then, by \cite[Theorem 2.4]{ramacher10} one has 
\bqn 
\hat \sigma_\chi (\rho e^{i(\cdot) \mu})=\lim_{\epsilon  \to 0} \sum_\gamma \left [ \frac{\mu^{n-1} d_\chi}{(2\pi)^{n-1}}  I^0_{\gamma,\epsilon }(\mu,1,0) + \mathfrak{R}_{\gamma,\epsilon }(\mu) \right ], \quad 
  \mathfrak{R}_{\gamma,\epsilon}(\mu)=O\Big (d_\chi\, \mu^{n-2} \sum_{|\beta| \leq 5} \sup _{R,t} | I^\beta_{\gamma,\epsilon}(\mu,R,t)|\Big ),
\eqn
where
\begin{align*}
I^\beta_{\gamma,\epsilon }(\mu,R,t):=  &   \int_{T^\ast U_\gamma }  \int _G  e^{i {\mu}\Phi_\gamma(x,\eta,g)}    \overline{\chi(g)} f_\gamma(x)  \bar f _\gamma (g  x) J_\gamma(g,  x)   \alpha( q( x, \eta))  \\  &\cdot \gd _{R,t}^\beta  \big [ \rho(t)  a_\gamma(t, \kappa_\gamma( x) , \mu \eta)     
  \Delta_{\epsilon,R}(\zeta_\gamma(t,\kappa_\gamma( x), \eta)) \big ] \, \d g  {\d(T^\ast U_\gamma)(x, \eta) },
\end{align*}
and  $ \Phi_\gamma(x,\eta,g)=\eklm{\kappa_\gamma( x) - \kappa_\gamma( g \cdot x),\eta}$. Here $a_\gamma \in S^0_\text{phg}$ are classical symbols with $a_\gamma(0,\kappa_\gamma(x),\eta)=1$, and $\zeta_\gamma$  certain smooth functions homogeneous in $\eta$ of degree $1$ satisfying $\zeta_\gamma(0, \kappa_\gamma(x),\eta)=q(x,\eta)$, while $J_\gamma(x,g)$ is a Jacobian. Based on the remainder estimate \eqref{eq:11.07.2015} it was then shown in \cite[Proposition 9.6]{ramacher10} that 
 \bq
 \label{eq:6.1.2014}
\lim_{\epsilon \to 0}  I^\beta_{\gamma,\epsilon }(\mu,R,t)=\Lcal_{\gamma}^\beta(R,t) (2\pi/\mu)^{\kappa}+ O_\chi(\mu^{-\kappa-1}(\log \mu )^{\Lambda -1}),
\eq
where the coefficients $\Lcal_{\gamma}^\beta(R,t)$ are  given in terms of distributions supported on the regular part of the critical set   $\Ccal \cap T^\ast U_\gamma$ of $\Phi_\gamma$ intersected with $G\times S^\ast_{t,R} U_\gamma$, and the remainder term by distributions supported on $G\times  S^\ast_{t,R} U_\gamma$,  where $S^\ast_{t,R} U_\gamma:=  \mklm{(x,\omega) \in T^\ast U_\gamma: \zeta_\gamma(t,\kappa_\gamma(x),\omega)=R}$. In particular, $\Lcal^0_\gamma (1,0)=[{\pi_\chi}_{|H}:\1] \, \rho(0) \, \mbox{vol} \, [({\mathrm{Reg}}\, \Omega \cap S^\ast M)/G]$. This yielded the estimate 
\begin{gather*}
\hat \sigma_\chi (\rho e^{i(\cdot) \mu})
=  d_\chi [{\pi_\chi}_{|H}:\1] \, \rho(0) \, \mbox{vol} \, [({\mathrm{Reg}}\, \Omega \cap S^\ast M)/G] \,  ({\mu}{/2\pi} )^{n-\kappa -1}  + O_\chi\big (\mu^{n-\kappa-2}(\log \mu )^{\Lambda -1} \big )
\end{gather*}
  from which an asymptotic description for $N_\chi(\lambda)$, and consequently $\M_\chi(\lambda)$,  was obtained via a classical Tauberian argument, see \cite[Proof of Theorem 9.5]{ramacher10}.
Now, in order to prove Theorem \ref{thm:weylfam}, note that by the improved bounds \eqref{eq:12.07.2015} of Theorem \ref{thm:main} the integrals $I^\beta_{\gamma,\epsilon }(\mu,R,t)$ actually have asymptotic descriptions with leading terms of order $\mu^{-\kappa}$ and remainder terms bounded from above by
\bqn 
C^\beta_{\gamma,\epsilon}(R,t) \sup_{l \leq 2\kappa+3} \norm{D^l \overline \chi}_{\infty}  \mu^{-\kappa-1}(\log\mu)^{\Lambda-1}
\eqn
with constants $C^\beta_{\gamma,\epsilon}(R,t)>0$ independent of $\chi$ and $\mu$. In analogy to   \eqref{eq:6.1.2014} one therefore deduces for any $\chi \in \W_{\mu^m}$
 \bqn
  \lim_{\epsilon \to 0} { I^\beta_{\gamma,\epsilon }(\mu,R,t)}=  {\Lcal_{\gamma}^\beta(R,t) (2\pi/\mu)^{\kappa}} +  O( [{\pi_\chi}_{|H}:\1] \, \mu^{\vartheta (2\kappa+3)m}\mu^{-\kappa-1}(\log \mu )^{\Lambda -1}),
\eqn
yielding with $d_\chi [{\pi_\chi}_{|H}:\1]\geq 1$ the asymptotic formula
\begin{align*}
\frac 1{|\W_{\mu^m}|} \sum_{\chi \in \W_{\mu^m}} \frac{ \hat \sigma_\chi (\rho e^{i(\cdot) \mu})}{d_\chi [{\pi_\chi}_{|H}:\1]}&= \rho(0) \, \mbox{vol} \, [({\mathrm{Reg}}\, \Omega \cap S^\ast M)/G] \,  ({\mu}{/2\pi} )^{n-\kappa -1}  \\&+ O(\mu^{\vartheta(2\kappa+3)m}\mu^{n-\kappa-2}(\log \mu )^{\Lambda -1}).
\end{align*}
The assertion of the theorem now follows again from a classical Tauberian argument, compare \cite[Proof of Theorem 9.5]{ramacher10}.
\end{proof}

\subsection{Families of irreducible representations and the Cartan-Weyl classification} 
\label{subsec:irreps} In what follows, we shall apply Theorem  \ref{thm:weylfam} to specific families of representations given in terms of the Cartan-Weyl classification  of unitary irreducible representations of $G$.  Thus \cite{wallach}, let $G$ be a connected compact Lie group with Lie algebra $\g$ and $T\subset G$ a maximal torus with Lie algebra $\t$. The exponential function $\exp$ is a covering homomorphism of $\t$ onto $T$, and its kernel $L$ a lattice in $\t$. Let $\widehat T$ denote the  \emph{set of characters of $T$}, that is, of all continuous homomorphisms of $T$ into the circle. The differential of a character $\mu: T \to S^1$, denoted by the same letter, is a linear form  $\mu:\t\to i \R$ which is \emph{integral} in the sense that $\mu(L) \subset 2\pi i \Z$. On the other hand, if $\mu$ is an integral linear form, one defines
\bqn 
t^\mu:= e^{\mu(X)}, \qquad t= \exp X \in T,
\eqn
setting up an identification of $\widehat T$ with the integral linear forms on $\t$. Let $\g_\C$ and $\t_\C$ denote the complexifications of $\t$ and $\g$, respectively. Then $\t_\C$ is a Cartan subalgebra of $\g_\C$, and we denote the corresponding system  of roots by $\Sigma(\g_\C,\t_\C)$.  Let $\Sigma^+$ denote a set of positive roots. Since roots define integral linear forms on $\t$, one can regard them as characters of $T$. 

As before, let $\widehat G$ be the set of all equivalence classes of irreducible unitary representations of $G$ and $\chi \in \widehat G$. Due to the invariance of the trace under cyclic permutations the character of $\chi$ satisfies  $\chi(t)=\chi(g t g^{-1})$ for all $t,g \in G$. Since any element in $G$ is conjugated to an element of $T$, $\chi(g)$ is fully determined by its restriction to $T$. Now, as a consequence of the Cartan-Weyl classification of irreducible finite dimensional representations of reductive Lie algebras over $\C$ one has the  identification 
\bq
\label{eq:unitdual}
\widehat G\, \simeq\, \mklm{\Lambda \in \t_\C^\ast: \text{ $\Lambda$ is  dominant integral and $T$-integral}}.
\eq 
Here an element $\Lambda \in \t_\C^\ast$ is called \emph{dominant integral} if  $2(\Lambda, \alpha)/(\alpha,\alpha)$ is a non-negative integer for any $\alpha \in \Sigma^+$, $(\cdot, \cdot)$ being the symmetric non-degenerate form on $\t_\C^\ast $ induced by an $\Ad(G)$-invariant inner product on $\g$.  

Next, assume that $G$ is semi-simple,  write $\rho:= \frac 12 \sum_{\alpha \in \Sigma^+} \alpha$, and let $W:=W(G,T)$ be the Weyl group. For an integral linear form $\mu\in \widehat T$, define the {alternating sum} $
A(\mu)(t):= \sum_{w \in W} \det w \,  t^{w\mu}$. If $\chi \in \widehat G$, let $\Lambda_\chi\in \t_\C^\ast$ be the \emph{highest weight} given by the isomorphism \eqref{eq:unitdual}. Then, the  Weyl character formula asserts that on $T$ one has \cite{varadarajan74}
\bqn 
\chi_{|T} (t)= \frac{A(\Lambda_\chi+\rho)(t)}{t^\rho \prod_{\alpha \in \Sigma^+}(1-t^{-\alpha})}= \frac{A(\Lambda_\chi+\rho)(t)}{ \prod_{\alpha \in \Sigma^+}(t^{\alpha/2}-t^{-\alpha/2})}.
\eqn
If $G$ is simply connected this can be written as $\chi_{|T} (t)= {A(\Lambda_\chi+\rho)(t)}/{A(\rho)(t)}$. 
Since $\mu(H) \in i\R$ for all $H \in \t$ and any integral linear form $\mu$ on $\t$, one immediately deduces for any $l \in \N$ as $|\Lambda_\chi| \to \infty$ the estimate
\bqn 
\frac{d^l}{ds^l} \chi_{|T}(\exp sH)= O( |\Lambda_\chi|^l).
\eqn
 Writing $g=h(g) t(g) h (g)^{-1}$ for an arbitrary element $g \in G$ with $t(g) \in T$ and $h(g) \in G$ we obtain with $\chi(g)= \chi_{|T}(t(g))$ the following simple consequence. If $D^l$  is a differential operator on $G$ of order $l$ and $\chi \in \widehat G$ a class with highest weight $\Lambda_\chi$,  then \bq
 \label{eq:29.08.2015}
 \norm{D^l \chi}_\infty = O(|\Lambda_\chi |  ^l), \qquad |\Lambda_\chi|\to \infty.  
\eq
From this we deduce the following

\begin{cor}
In the setup of Section \ref{sec:famirreps} define $\mathcal{W}_\lambda:=\big\{\chi \in \widehat G': \, |\Lambda_\chi|\leq C \, \lambda^{\vartheta}\big\}$ for some $\vartheta\in \big [0, \frac 1{(2\kappa+3)m}\big )$ and a constant $C>0$. Then  $\{\W_\lambda\}_{\lambda\in(0,\infty)}$ constitutes a  family  with growth rate $\vartheta$, and Theorem \ref{thm:weylfam} applies.
% as $\lambda \to \infty$ one has the asymptotic formula 
%\bqn
%\frac 1{|\W_\lambda|}\sum_{\chi \in \widehat G', |\Lambda_\chi|\leq \lambda ^\vartheta} \frac{\M_\chi(\lambda)}{[{\pi_\chi}_{|H}:\1]}= \frac{\mathrm{vol} \, [(\Omega \cap S^\ast M)/G]}{(n-\kappa)(2\pi)^{n-\kappa}}    
% \,  {\lambda} ^{\frac{n-\kappa}m }   + O\big (\lambda^{\frac{n-\kappa-1}m+ \vartheta(2\kappa+3)} (\log \lambda )^\Lambda \big ).
%\eqn
\end{cor}
\begin{proof}
By \eqref{eq:29.08.2015} we have for each $l\in \N$ and each differential operator $D^l$ on $G$ of order $l$
 \[
\max_{\chi\in \W_\lambda}\frac{\norm{D^l\overline{\chi}}_\infty}{\left[\pi_{\chi|_{H}}:\mathds{1}\right]}\leq \max_{\chi\in \W_\lambda}{\norm{D^l\overline{\chi}}_\infty}{} \leq C\,\lambda^{\vartheta l }\qquad \forall \;\lambda \in(0,\infty)
\]
with a constant $C>0$ independent of $\lambda$.\footnote{Note that  the Weyl character formula implies the dimension formula $d_\chi=\prod_{\alpha \in \Sigma^+} (\Lambda_\chi+\rho,\alpha)/(\alpha,\alpha)$. 
Thus, for the set $\gd \mathcal{W}_\lambda:=\mklm{\chi \in \widehat G': \, |\Lambda_\chi|= \lambda^{\vartheta}}$ one gets the stronger estimate
\[
\max_{\chi\in \gd \W_\lambda}\frac{\norm{D^l\overline{\chi}}_\infty}{d_{\chi}\left[\pi_{\chi|_{H}}:\mathds{1}\right]}  \leq C\,\lambda^{\vartheta (l-|\Sigma^+|) }\qquad \forall \;\lambda\in(0,\infty).
\]}

\end{proof}

% Koppelow, el 12.08.2015 y Villa Adelina, el 01.09.2015.

%******************* Revision final hasta aqui ****************

\providecommand{\bysame}{\leavevmode\hbox to3em{\hrulefill}\thinspace}
\providecommand{\MR}{\relax\ifhmode\unskip\space\fi MR }
\providecommand{\MRhref}[2]{%
  \href{http://www.ams.org/mathscinet-getitem?mr=#1}{#2}
}
\providecommand{\href}[2]{#2}

%\bibliography{bibliography}
%\bibliographystyle{amsplain}

\end{document}